\documentclass[12pt,twoside,a4paper]{amsart}

\usepackage{amssymb,amsmath,amsthm}

\usepackage{color}

\newtheorem{theorem}{Theorem}[section]
\newtheorem{thm}[theorem]{Th\'eor\`eme}

\newtheorem{coro}[theorem]{Corollaire}

\theoremstyle{plain}
\newtheorem*{namedthm}{\namedthmname}
\newcounter{namedthm}



\newcommand{\R}{\mathbb{R}}
\newcommand{\C}{\mathbb{C}}
\newcommand{\N}{\mathbb{N}}

\numberwithin{equation}{section}

\usepackage{hyperref}
\hypersetup{
	bookmarks=true,         
	unicode=false,         
	pdftoolbar=true,       
	pdfmenubar=true,       
	pdffitwindow=false,    
	pdfstartview={FitH},   
	pdftitle={Pluripotential vs Viscosity solutions to complex Monge-Amp\`ere flows},    
	pdfauthor={Vincent Guedj, Chinh H. Lu, and Ahmed Zeriahi},     
	colorlinks=true,       
	linkcolor=black,         
	citecolor=black,        
	filecolor=black,      
	urlcolor=black}          

\begin{document}
	\def\K{\mathbb{K}}
	\def\R{\mathbb{R}}
	\def\C{\mathbb{C}}
	\def\Z{\mathbb{Z}}
	\def\Q{\mathbb{Q}}
	\def\D{\mathbb{D}}
	\def\N{\mathbb{N}}
	\def\T{\mathbb{T}}
	\def\P{\mathbb{P}}
	\def\A{\mathscr{A}}
	\def\CC{\mathscr{C}}
	\renewcommand{\theequation}{\thesection.\arabic{equation}}
	\renewenvironment{proof}{{\bfseries Preuve:}}{\qed}
	\renewcommand{\thelemme}{\empty{}}
	\newtheorem{preuv}{Preuve}[section]
	\newtheorem{cond}{C}
	\newtheorem{lemma}{Lemme}[section]
	\newtheorem{corollary}{Corollary}[section]
	\newtheorem{proposition}{Proposition}[section]
	\newtheorem{notation}{Notation}[section]
	\newtheorem{remark}{Remark}[section]
	\newtheorem{example}{Example}[section]
	\newtheorem{probleme}{Probleme}[section]
	\bibliographystyle{plain}

\title[Résultats sur quelques op{\'e}rateurs ~~]{\textbf {Résultats sur quelques op{\'e}rateurs  dans l'espace de Hilbert \`a poids} $L^2(\mathbb{\C}, \mathrm{\textnormal{e}}^{\mathrm{-\vert z \vert^2}})$.}
\author[S.\  Sambou  ]{ Souhaibou Sambou  }
\address{D\'epartement de Math\'ematiques\\UFR des Sciences et Technologies \\ Universit\'e Assane Seck de Ziguinchor, BP: 523 (S\'en\'egal)}
\email{s.sambou1440@zig.univ.sn }
 

\subjclass{}

\maketitle
\renewcommand{\abstractname}{Résumé}
\begin{abstract}
Par la $L^2$-méthode de H\"ormander, on étudie quelques op{\'e}rateurs dans l'espace de Hilbert à poids $L^2(\mathbb{\C}, \mathrm{\textnormal{e}}^{-|z|^2})$. On prouve dans chaque cas d'opérateur l'existence  de son inverse à droite qui est aussi un opérateur borné.  

\vskip 2mm
\noindent
\keywords{{\bf Mots clés:}  Les opérateurs $\alpha \Delta + c$, \;  $\alpha D^{k} + c$, $\alpha \bar{\partial}^{k} + c$ et $\alpha \partial^k \bar{\partial}^{k} + c$, espace de Hilbert à poids $L^2(\mathbb{\C}, \mathrm{\textnormal{e}}^{-|z|^2})$, $L^2$-méthode de H\"ormander.}
\vskip 1.3mm
\noindent
{\bf Mathematics Subject Classification (2010)} . 32F32.
\end{abstract}  
\renewcommand{\abstractname}{Abstract}
\section{Introduction}
De ce travail, on s'intéresse à une série de conséquences de quelques opérateurs réels et complexes.\\
Dans le cas réel, on s'intéresse en première partie à l'étude de l'opérateur $\alpha \Delta  + c$ o\`u  $\alpha$ est une constante réel tel que $\vert \alpha \vert \geq 1$ dans l'espace $L^2(\mathbb{\R}, \mathrm{\textnormal{e}}^{-  x ^2})$. On a ainsi le résultat suivant:
\begin{thm} \label{b}
Pour tout $f \in L^2(\mathbb{\R}, \mathrm{\textnormal{e}}^{-  x ^2})$, il existe une solution faible $u \in L^2(\mathbb{\R}, \mathrm{\textnormal{e}}^{-x^2})$ de l'équation
$$\alpha \Delta u + cu = \alpha f$$ avec $\vert \alpha \vert \geq 1$ et l'estimation de la norme
$$\int_\mathbb{R}  u ^2 \mathrm{\textnormal{e}}^{-  x ^2} dx \leq \frac{ \alpha ^2}{8} \int_\mathbb{R}  f ^2 \mathrm{\textnormal{e}}^{-x^2} dx.$$
\end{thm}
 Ce résultat est une conséquence du Théorème$1.1$ de \cite{3}.\\
 En deuxième partie, on s'intéresse à l'étude de l'opérateur $\alpha D^{k}  + c$ o\`u  $\alpha$ est une constante réel tel que $\vert \alpha \vert \geq 1$ et $D^{k} = \frac{d^k}{dx^k}$ dans  $L^2(\mathbb{\R}, \mathrm{\textnormal{e}}^{-  x ^2})$. Le résultat est le suivant:
 \begin{thm} \label{a}
Pour tout $f \in L^2(\mathbb{\R}, \mathrm{\textnormal{e}}^{-  x ^2})$, il existe une solution faible $u \in L^2(\mathbb{\R}, \mathrm{\textnormal{e}}^{-x^2})$ de l'équation
$$\alpha D^{k} u + cu = \alpha f$$ avec $\vert \alpha \vert \geq 1$ et l'estimation de la norme
$$\int_\mathbb{R}  u ^2 \mathrm{\textnormal{e}}^{-  x ^2} dx \leq \frac{ \alpha ^2}{2^k k!} \int_\mathbb{R}  f ^2 \mathrm{\textnormal{e}}^{-x^2} dx.$$
\end{thm}
Ce résultat est une conséquence du Théorème$1.1$ de \cite{1}.\\
Dans le complexe, on s'intéresse en première partie à l'étude de l'opérateur $\alpha \bar{\partial}^{k}  + c$ o\`u  $\alpha$ est une constante complexe tel que $\vert \alpha \vert \geq 1$ et $\bar{\partial}^{k} = \frac{\partial^k}{\partial \bar{z}^k}$ dans  $L^2(\mathbb{\C}, \mathrm{\textnormal{e}}^{- \vert z \vert^2})$. Le résultat est le suivant:

\begin{thm} \label{q}
Pour tout $f \in L^2(\mathbb{\C}, \mathrm{\textnormal{e}}^{- \vert z \vert^2})$, il existe une solution faible $u \in L^2(\mathbb{\C}, \mathrm{\textnormal{e}}^{- \vert z \vert^2})$ de l'équation
$$\alpha \bar{\partial}^{k} u + cu = \alpha f$$ avec  $\vert \alpha \vert \geq 1$ et l'estimation de la norme
$$\int_\mathbb{C} \vert u \vert^2 \mathrm{\textnormal{e}}^{- \vert z \vert^2} d\sigma \leq \frac{\vert \alpha \vert^2}{(k!)} \int_\mathbb{C} \vert f \vert^2 \mathrm{\textnormal{e}}^{- \vert z \vert^2} d\sigma$$
\end{thm}
Ce résultat est une conséquence du Théorème$1.1$ de \cite{2}.\\
Et pour finir en dernière partie par l'étude de l'opérateur $\alpha \partial^{k} \bar{\partial}^{k}  + c$
o\`u  $\alpha$ est une constante complexe tel que $\vert \alpha \vert \geq 1$, $\bar{\partial}^{k} = \frac{\partial^k}{\partial \bar{z}^k}$ et $\partial^{k} = \frac{\partial^k}{\partial z^k}$ dans  $L^2(\mathbb{\C}, \mathrm{\textnormal{e}}^{- \vert z \vert^2})$. Le résultat est le suivant:

\begin{thm} \label{p}
Pour tout $f \in L^2(\mathbb{\C}, \mathrm{\textnormal{e}}^{- \vert z \vert^2})$, il existe une solution faible $u \in L^2(\mathbb{\C}, \mathrm{\textnormal{e}}^{- \vert z \vert^2})$ de l'équation
$$\alpha \partial^{k} \bar{\partial}^{k} u + cu = \alpha f$$ avec $\vert \alpha \vert \geq 1$ et l'estimation de la norme
$$\int_\mathbb{C} \vert u \vert^2 \mathrm{\textnormal{e}}^{- \vert z \vert^2} d\sigma \leq \frac{\vert \alpha \vert^2}{(k!)^2} \int_\mathbb{C} \vert f \vert^2 \mathrm{\textnormal{e}}^{- \vert z \vert^2} d\sigma$$
\end{thm}
Ce résultat est une conséquence du Théorème$1.1$ de \cite{4}.\\
\section{Cas réel}
\subsection{L'opérateur $\alpha \Delta  + c$} 

\begin{proof}{(Théorème \ref{b})}
Soit $f \in L^2(\mathbb{\R}, \mathrm{\textnormal{e}}^{- x^2})$, d'après le Théorème $1.1$ de \cite{3}, il existe une solution faible $u \in L^2(\mathbb{\R}, \mathrm{\textnormal{e}}^{- x^2})$ de l'équation
$$ \Delta u + c' u =  f$$ o\`u $c'= \frac{c}{\alpha}$ avec l'estimation de la norme
$$\int_\mathbb{R}  u ^2 \mathrm{\textnormal{e}}^{- x^2} dx \leq \frac{1}{8} \int_\mathbb{R} f^2 \mathrm{\textnormal{e}}^{- x^2} dx$$ 
donc $$\alpha  \Delta u + cu = \alpha f.$$
Or $ \alpha^2 \geq 1$ donc
$$\frac{1}{8} \int_\mathbb{R}  f^2 \mathrm{\textnormal{e}}^{- x^2} dx \leq \frac{ \alpha^2}{8} \int_\mathbb{R} f^2 \mathrm{\textnormal{e}}^{- x^2} dx$$
Ainsi
$$\int_\mathbb{R} u^2 \mathrm{\textnormal{e}}^{-x^2} dx \leq \frac{\alpha^2}{8} \int_\mathbb{R} f^2 \mathrm{\textnormal{e}}^{- x^2} dx$$
\end{proof}
\begin{theorem}
Il existe un opérateur linéaire et borné $$T : L^2(\mathbb{\R}, \mathrm{\textnormal{e}}^{- x^2}) \longrightarrow L^2(\mathbb{\R}, \mathrm{\textnormal{e}}^{- x^2})$$
tel que $$ (\alpha \Delta  + c) T = I$$ avec l'estimation de la norme 
$$\vert \vert T \vert \vert_\varphi \leq \frac{1}{\sqrt{8}}$$ o\`u $\vert \vert T \vert \vert_\varphi$ est la norme de $T$ dans $L^2(\mathbb{\R}, \mathrm{\textnormal{e}}^{- x^2})$.
\end{theorem}
\begin{proof}
Soit $f \in L^2(\mathbb{\R}, \mathrm{\textnormal{e}}^{- x^2})$. D'après le théorème \ref{b}, il existe une solution faible $u \in L^2(\mathbb{\R}, \mathrm{\textnormal{e}}^{- x^2})$ telle que
$$\alpha  \Delta u + cu = \alpha f$$  avec l'estimation de la norme 
 $$\vert \vert u \vert \vert_\varphi \leq \frac{\vert \alpha \vert}{\sqrt{8}} \vert \vert f \vert \vert_\varphi$$
 Or $\alpha f \in L^2(\mathbb{\R}, \mathrm{\textnormal{e}}^{- x^2})$\\
donc 
$$ (\alpha \Delta  + c)T(\alpha f) = \alpha f$$ avec 
 $$\vert \vert T(\alpha f) \vert \vert_\varphi \leq \frac{\vert \alpha \vert}{\sqrt{8}} \vert \vert f \vert \vert_\varphi.$$
 Ainsi
 $$ (\alpha \Delta   + c)T = I$$ avec 
 $$\vert \vert T \vert \vert_\varphi \leq \frac{1}{\sqrt{8}}.$$
\end{proof}

\subsection{Quelques conséquences du cas de l'opérateur $\alpha \Delta  + c$} 

Soient $\lambda > 0$ et $x_0 \in \mathbb{C}$. Posons $\varphi = \lambda (x - x_0 )^2$, on obtient le corollaire du Théorème \ref{b} suivant:
\begin{coro} \label{lam2}
Soit $f \in L^2(\mathbb{\R}, \mathrm{\textnormal{e}}^{-\lambda (x - x_0 )^2})$, il existe une solution faible $u \in L^2(\mathbb{\R}, \mathrm{\textnormal{e}}^{-\lambda (x - x_0 )^2})$ de l'équation
$$\alpha  \Delta u + cu = \alpha f$$ avec $\vert \alpha \vert \geq 1$ et l'estimation de la norme
$$\int_\mathbb{R}  u^2 \mathrm{\textnormal{e}}^{- \lambda (x - x_0 )^2} dx \leq \frac{\alpha^2}{8\lambda^2} \int_\mathbb{R}  f^2 \mathrm{\textnormal{e}}^{- \lambda (x - x_0 )^2} dx$$
\end{coro}
\begin{proof}
Soit $f \in L^2(\mathbb{\R}, \mathrm{\textnormal{e}}^{-\lambda (x - x_0 )^2})$, on a $$\int_\mathbb{R}  f^2(x) \mathrm{\textnormal{e}}^{- \lambda ( x - x_0 )^2} dx< + \infty.$$
Posons $x = \frac{y}{\sqrt{\lambda}} + x_0$ et $g(y)= f(x)= f(\frac{y}{\sqrt{\lambda}} + x_0)$ alors
$$\frac{1}{\sqrt{\lambda}} \int_\mathbb{R}  g^2(y) \mathrm{\textnormal{e}}^{- y^2} dy < + \infty.$$
donc $g \in L^2(\mathbb{\R}, \mathrm{\textnormal{e}}^{- y^2})$. D'après le Théorème \ref{b}, il existe une solution faible $v \in L^2(\mathbb{\R}, \mathrm{\textnormal{e}}^{- y^2})$ de l'équation
$$\alpha \Delta v + \frac{c}{\lambda}v = \alpha g$$ avec l'estimation de la norme
$$\int_\mathbb{R}  v^2(y) \mathrm{\textnormal{e}}^{- y^2} dy \leq \frac{\alpha^2}{8} \int_\mathbb{R} g^2(y) \mathrm{\textnormal{e}}^{- y^2} dy.$$
Posons $u(x) = \frac{1}{\lambda} v(y) = \frac{1}{\lambda} v(\sqrt{\lambda} (x - x_0))$.\\
Alors 
$$\alpha \Delta u + cu = \alpha f$$ avec l'estimation de la norme
$$\int_\mathbb{R}  u^2 \mathrm{\textnormal{e}}^{- \lambda (x - x_0)^2} dx \leq \frac{\alpha^2}{8\lambda^2} \int_\mathbb{R}  f^2 \mathrm{\textnormal{e}}^{- \lambda^2 ( x - x_0)^2} dx.$$
\end{proof}

Comme conséquence du Corollaire \ref{lam1} on a
\begin{coro} \label{3}
Soit $U \in \mathbb{R}$ un ouvert borné.\\
Soit $f \in L^2(U)$, il existe une solution faible $u \in L^2(U)$ de l'équation
$$\alpha \Delta u + cu = \alpha f$$ avec $\vert \alpha \vert \geq 1$ et
$$\vert \vert u \vert \vert_{L^2(U)} \leq \Big(\frac{{\textnormal{e}}^{\vert U \vert^2} \alpha^2}{8}\Big) \vert \vert f \vert \vert_{L^2(U)}$$ 
o\`u $\vert U \vert $ est le diamètre de $U$.
\end{coro}
\begin{proof}
Soient $x_0 \in U$ et $f \in L^2(U)$, on a
 \[ 
 \tilde{f}(x) = \left \{
 \begin{array}{rl}
 f(x) & \mbox{ si } x \in U \\
 0 & \mbox{ si } x \in \mathbb{R} \setminus U
 \end{array}
 \right.
 \]
 Donc $\tilde{f} \in L^2(\mathbb{R}) \subset L^2(\mathbb{\R}, \mathrm{\textnormal{e}}^{-( x - x_0)^2})$. D'après le Corollaire \ref{lam1}, il existe une solution faible $\tilde{u} \in L^2(\mathbb{\R}, \mathrm{\textnormal{e}}^{-( x - x_0)^2})$ de l'équation $$\alpha  \Delta \tilde{u} + c\tilde{u} = \alpha \tilde{f}$$ avec
 $$\int_\mathbb{R}  \tilde{u}^2 \mathrm{\textnormal{e}}^{-( x - x_0)^2} dx \leq \frac{\alpha^2}{8} \int_\mathbb{R}  \tilde{f}^2 \mathrm{\textnormal{e}}^{ -( x - x_0)^2} dx$$
 donc
 $$\int_\mathbb{R} \tilde{u}^2 \mathrm{\textnormal{e}}^{-( x - x_0)^2} dx \leq \frac{\alpha^2}{8} \int_U  f^2  dx$$
 or $$\int_\mathbb{R} \tilde{u}^2 \mathrm{\textnormal{e}}^{-( x - x_0)^2} dx \geq \mathrm{\textnormal{e}}^{-\vert U \vert^2} \int_\mathbb{U}  \tilde{u}^2  dx$$
 ainsi
 $$\int_U \tilde{u}^2  dx \leq \Big(\frac{{\textnormal{e}}^{\vert U \vert^2} \alpha^2}{8}\Big) \int_U  f^2  dx$$
 or $\tilde{u}_{\mid U} = u$ alors on a
 $$\alpha \Delta u + cu = \alpha f$$ avec 
$$\vert \vert u \vert \vert_{L^2(U)} \leq \Big(\frac{{\textnormal{e}}^{\vert U \vert^2} \alpha^2}{8}\Big)\vert \vert f \vert \vert_{L^2(U)}.$$ 
\end{proof}

\subsection{L'opérateur $\alpha D^{k}  + c$}

\begin{proof}{(Théorème \ref{a})}
Soit $f \in L^2(\mathbb{\R}, \mathrm{\textnormal{e}}^{- x^2})$, d'après le théorème $1.1$ de \cite{1}, il existe une solution faible $u \in L^2(\mathbb{\R}, \mathrm{\textnormal{e}}^{- x^2})$ de l'équation
$$ D^{k} u + c' u =  f$$ o\`u $c'= \frac{c}{\alpha}$ avec l'estimation de la norme
$$\int_\mathbb{R}  u ^2 \mathrm{\textnormal{e}}^{- x^2} dx \leq \frac{1}{2^k k!} \int_\mathbb{R} f^2 \mathrm{\textnormal{e}}^{- x^2} dx$$ 
donc $$\alpha  D^{k} u + cu = \alpha f.$$
Or $ \alpha^2 \geq 1$ donc
$$\frac{1}{2^k k!} \int_\mathbb{R}  f^2 \mathrm{\textnormal{e}}^{- x^2} dx \leq \frac{ \alpha^2}{2^k k!} \int_\mathbb{R} f^2 \mathrm{\textnormal{e}}^{- x^2} dx$$
Ainsi
$$\int_\mathbb{R} u^2 \mathrm{\textnormal{e}}^{-x^2} dx \leq \frac{\alpha^2}{2^k k!} \int_\mathbb{R} f^2 \mathrm{\textnormal{e}}^{- x^2} dx$$
\end{proof}
\begin{theorem}
Il existe un opérateur linéaire et borné $$T_k : L^2(\mathbb{\R}, \mathrm{\textnormal{e}}^{- x^2}) \longrightarrow L^2(\mathbb{\R}, \mathrm{\textnormal{e}}^{- x^2})$$
tel que $$ (\alpha D^{k}  + c) T_k = I$$ avec l'estimation de la norme 
$$\vert \vert T_k \vert \vert_\varphi \leq \frac{1}{\sqrt{2^k k!}}$$ o\`u $\vert \vert T_k \vert \vert_\varphi$ est la norme de $T_k$ dans $L^2(\mathbb{\R}, \mathrm{\textnormal{e}}^{- x^2})$.
\end{theorem}
\begin{proof}
Soit $f \in L^2(\mathbb{\R}, \mathrm{\textnormal{e}}^{- x^2})$. D'après le théorème \ref{a}, il existe une solution faible $u \in L^2(\mathbb{\R}, \mathrm{\textnormal{e}}^{- x^2})$ telle que
$$\alpha  D^{k} u + cu = \alpha f$$  avec l'estimation de la norme 
 $$\vert \vert u \vert \vert_\varphi \leq \frac{\vert \alpha \vert}{\sqrt{2^k k!}} \vert \vert f \vert \vert_\varphi$$
 Or $\alpha f \in L^2(\mathbb{\R}, \mathrm{\textnormal{e}}^{- x^2})$\\
donc 
$$ (\alpha D^{k}  + c)T_k(\alpha f) = \alpha f$$ avec 
 $$\vert \vert T_k(\alpha f) \vert \vert_\varphi \leq \frac{\vert \alpha \vert}{\sqrt{2^k k!}} \vert \vert f \vert \vert_\varphi.$$
 Ainsi
 $$ (\alpha D^{k}  + c)T_k = I$$ avec 
 $$\vert \vert T_k \vert \vert_\varphi \leq \frac{1}{\sqrt{2^k k!}}.$$
\end{proof}
\begin{thm}
Soit $\varphi \in C^{\infty}(\mathbb{R})$ une fonction strictement convexe. Pour tout $ f \in L^2(\mathbb{\R}, \mathrm{\textnormal{e}}^{-\varphi})$ telle que $$\frac{f}{\sqrt{D^2 \varphi }}  \in L^2(\mathbb{\R}, \mathrm{\textnormal{e}}^{- \varphi}),$$ il existe une solution faible $ u \in  L^2(\mathbb{\R}, \mathrm{\textnormal{e}}^{-\varphi})$ de l'équation
$$\alpha  D^1 u + cu = \alpha f$$ avec  $\vert \alpha \vert \geq 1$ et l'estimation de la norme
$$\int_\mathbb{R} u^2 \mathrm{\textnormal{e}}^{-\varphi} dx \leq   \alpha^2 \int_\mathbb{R} \frac{f^2}{D^2 \varphi}  \mathrm{\textnormal{e}}^{-\varphi} dx.$$
\end{thm}
\begin{proof}
Soit $ f \in L^2(\mathbb{\R}, \mathrm{\textnormal{e}}^{-\varphi})$ telle que  $$\frac{f}{\sqrt{D^2 \varphi }}  \in L^2(\mathbb{\R}, \mathrm{\textnormal{e}}^{-\varphi}),$$
d'après le théorème $1.2$ de \cite{1}, il existe une solution faible $u \in L^2(\mathbb{\R}, \mathrm{\textnormal{e}}^{-\varphi})$ de l'équation
$$  D^1 u + c' u =  f$$ o\`u $c'= \frac{c}{\alpha}$ avec l'estimation de la norme 
$$\int_\mathbb{R}  u^2 \mathrm{\textnormal{e}}^{-\varphi} dx \leq  \int_\mathbb{R} \frac{f^2}{D^2 \varphi}  \mathrm{\textnormal{e}}^{-\varphi} dx.$$
donc $$\alpha  D^1 u + cu = \alpha f.$$
Or $\vert \alpha \vert \geq 1$ donc
$$  \int_\mathbb{R} \frac{f^2}{D^2 \varphi}  \mathrm{\textnormal{e}}^{-\varphi} dx \leq    \alpha^2 \int_\mathbb{R} \frac{f^2}{D^2 \varphi}  \mathrm{\textnormal{e}}^{-\varphi} dx.$$
Ainsi
$$\int_\mathbb{R} u^2 \mathrm{\textnormal{e}}^{-\varphi} dx \leq  \alpha^2 \int_\mathbb{R} \frac{f^2}{D^2 \varphi }  \mathrm{\textnormal{e}}^{-\varphi} dx.$$
\end{proof}
\subsection{Quelques conséquences du cas de l'opérateur $\alpha D^{k}  + c$}
Soient $\lambda > 0$ et $x_0 \in \mathbb{R}$. Posons $\varphi = \lambda (x - x_0 )^2$, on obtient le corollaire du Théorème \ref{a} suivant:
\begin{coro} \label{lam1}
Soit $f \in L^2(\mathbb{\R}, \mathrm{\textnormal{e}}^{-\lambda (x - x_0 )^2})$, il existe une solution faible $u \in L^2(\mathbb{\R}, \mathrm{\textnormal{e}}^{-\lambda (x - x_0 )^2})$ de l'équation
$$\alpha  D^k u + cu = \alpha f$$ avec $\vert \alpha \vert \geq 1$ et l'estimation de la norme
$$\int_\mathbb{R}  u^2 \mathrm{\textnormal{e}}^{- \lambda (x - x_0 )^2} dx \leq \frac{\alpha^2}{(2 \lambda)^kk!} \int_\mathbb{R}  f^2 \mathrm{\textnormal{e}}^{- \lambda (x - x_0 )^2} dx$$
\end{coro}
\begin{proof}
Soit $f \in L^2(\mathbb{\R}, \mathrm{\textnormal{e}}^{-\lambda (x - x_0 )^2})$, on a $$\int_\mathbb{R}  f^2(x) \mathrm{\textnormal{e}}^{- \lambda ( x - x_0 )^2} dx< + \infty.$$
Posons $x = \frac{y}{\sqrt{\lambda}} + x_0$ et $g(y)= f(x)= f(\frac{y}{\sqrt{\lambda}} + x_0)$ alors
$$\frac{1}{\sqrt{\lambda}} \int_\mathbb{R}  g^2(y) \mathrm{\textnormal{e}}^{- y^2} dy < + \infty.$$
donc $g \in L^2(\mathbb{\R}, \mathrm{\textnormal{e}}^{- y^2})$. D'après le Théorème \ref{a}, il existe une solution faible $v \in L^2(\mathbb{\R}, \mathrm{\textnormal{e}}^{- y^2})$ de l'équation
$$\alpha D^k v + \frac{c}{(\sqrt{\lambda})^k} v = \alpha g$$ avec l'estimation de la norme
$$\int_\mathbb{R}  v^2(y) \mathrm{\textnormal{e}}^{- y^2} dy \leq \frac{\alpha^2}{2^k k!} \int_\mathbb{R} g^2(y) \mathrm{\textnormal{e}}^{- y^2} dy.$$
Posons $u(x) = \frac{1}{(\sqrt{\lambda})^k} v(y) = \frac{1}{(\sqrt{\lambda})^k} v(\sqrt{\lambda} (x - x_0))$.\\
Alors 
$$\alpha D^k u + cu = \alpha f$$ avec l'estimation de la norme
$$\int_\mathbb{R}  u^2 \mathrm{\textnormal{e}}^{- \lambda (x - x_0)^2} dx \leq \frac{\alpha^2}{(2 \lambda)^k k!} \int_\mathbb{R}  f^2 \mathrm{\textnormal{e}}^{- \lambda^2 ( x - x_0)^2} dx.$$
\end{proof}

Comme conséquence du Corollaire \ref{lam1} on a
\begin{coro} \label{3}
Soit $U \in \mathbb{R}$ un ouvert borné.\\
Soit $f \in L^2(U)$, il existe une solution faible $u \in L^2(U)$ de l'équation
$$\alpha D^k u + cu = \alpha f$$ avec $\vert \alpha \vert \geq 1$ et
$$\vert \vert u \vert \vert_{L^2(U)} \leq \Big(\frac{{\textnormal{e}}^{\vert U \vert^2} \alpha^2}{(2 \lambda)^k k!}\Big) \vert \vert f \vert \vert_{L^2(U)}$$ 
o\`u $\vert U \vert $ est le diamètre de $U$.
\end{coro}
\begin{proof}
Soient $x_0 \in U$ et $f \in L^2(U)$, on a
 \[ 
 \tilde{f}(x) = \left \{
 \begin{array}{rl}
 f(x) & \mbox{ si } x \in U \\
 0 & \mbox{ si } x \in \mathbb{R} \setminus U
 \end{array}
 \right.
 \]
 Donc $\tilde{f} \in L^2(\mathbb{R}) \subset L^2(\mathbb{\R}, \mathrm{\textnormal{e}}^{-( x - x_0)^2})$. D'après le Corollaire \ref{lam1}, il existe une solution faible $\tilde{u} \in  L^2(\mathbb{\R}, \mathrm{\textnormal{e}}^{-( x - x_0)^2})$ de l'équation $$\alpha  D^k \tilde{u} + c\tilde{u} = \alpha \tilde{f}$$ avec
 $$\int_\mathbb{R}  \tilde{u}^2 \mathrm{\textnormal{e}}^{-( x - x_0)^2} dx \leq \frac{\alpha^2}{(2 \lambda)^k k!} \int_\mathbb{R}  \tilde{f}^2 \mathrm{\textnormal{e}}^{ -( x - x_0)^2} dx$$
 donc
 $$\int_\mathbb{R} \tilde{u}^2 \mathrm{\textnormal{e}}^{-( x - x_0)^2} dx \leq \frac{\alpha^2}{(2 \lambda)^k k!} \int_U  f^2  dx$$
 or $$\int_\mathbb{R} \tilde{u}^2 \mathrm{\textnormal{e}}^{-( x - x_0)^2} dx \geq \mathrm{\textnormal{e}}^{-\vert U \vert^2} \int_\mathbb{U}  \tilde{u}^2  dx$$
 ainsi
 $$\int_U \tilde{u}^2  dx \leq \Big(\frac{{\textnormal{e}}^{\vert U \vert^2} \alpha^2}{(2 \lambda)^k k!}\Big) \int_U  f^2  dx$$
 or $\tilde{u}_{\mid U} = u$ alors on a
 $$\alpha D^k u + cu = \alpha f$$ avec 
$$\vert \vert u \vert \vert_{L^2(U)} \leq \Big(\frac{{\textnormal{e}}^{\vert U \vert^2} \alpha^2}{(2 \lambda)^k k!}\Big)\vert \vert f \vert \vert_{L^2(U)}.$$ 
\end{proof}

\section{Cas complexe}
\subsection{L'opérateur $\alpha \bar{\partial}^{k}  + c$}

\begin{proof}{(Théorème \ref{q})}
Soit $f \in L^2(\mathbb{\C}, \mathrm{\textnormal{e}}^{- \vert z \vert^2})$, d'après le théorème $1.1$ de \cite{2}, il existe une solution faible $u \in L^2(\mathbb{\C}, \mathrm{\textnormal{e}}^{- \vert z \vert^2})$ de l'équation
$$ \bar{\partial}^{k} u + c' u =  f$$ o\`u $c'= \frac{c}{\alpha}$ avec l'estimation de la norme
$$\int_\mathbb{C} \vert u \vert^2 \mathrm{\textnormal{e}}^{- \vert z \vert^2} d\sigma \leq \frac{1}{(k!)} \int_\mathbb{C} \vert f \vert^2 \mathrm{\textnormal{e}}^{- \vert z \vert^2} d\sigma$$ 
donc $$\alpha  \bar{\partial}^{k} u + cu = \alpha f.$$
Or $\vert \alpha \vert \geq 1$ donc
$$\frac{1}{(k!)} \int_\mathbb{C} \vert f \vert^2 \mathrm{\textnormal{e}}^{- \vert z \vert^2} d\sigma \leq \frac{\vert \alpha \vert^2}{(k!)} \int_\mathbb{C} \vert f \vert^2 \mathrm{\textnormal{e}}^{- \vert z \vert^2} d\sigma$$
Ainsi
$$\int_\mathbb{C} \vert u \vert^2 \mathrm{\textnormal{e}}^{- \vert z \vert^2} d\sigma \leq \frac{\vert \alpha \vert^2}{(k!)} \int_\mathbb{C} \vert f \vert^2 \mathrm{\textnormal{e}}^{- \vert z \vert^2} d\sigma$$
\end{proof}
\begin{theorem}
Il existe un opérateur linéaire et borné $$T_k : L^2(\mathbb{\C}, \mathrm{\textnormal{e}}^{- \vert z \vert^2}) \longrightarrow L^2(\mathbb{\C}, \mathrm{\textnormal{e}}^{- \vert z \vert^2})$$
tel que $$ (\alpha \bar{\partial}^{k}  + c) T_k = I$$ avec $\vert \alpha \vert \geq 1$ et l'estimation de la norme 
$$\vert \vert T_k \vert \vert_\varphi \leq \frac{1}{(k!)}$$ o\`u $\vert \vert T_k \vert \vert_\varphi$ est la norme de $T_k$ dans $L^2(\mathbb{\C}, \mathrm{\textnormal{e}}^{- \vert z \vert^2})$.
\end{theorem}
\begin{proof}
Soit $f \in L^2(\mathbb{\C}, \mathrm{\textnormal{e}}^{- \vert z \vert^2})$. D'après le théorème \ref{q}, il existe une solution faible $u \in L^2(\mathbb{\C}, \mathrm{\textnormal{e}}^{- \vert z \vert^2})$ telle que
$$\alpha  \bar{\partial}^{k} u + cu = \alpha f$$  avec l'estimation de la norme 
 $$\vert \vert u \vert \vert_\varphi \leq \frac{\vert \alpha \vert^2}{(k!)} \vert \vert f \vert \vert_\varphi$$
 Or $\alpha f \in L^2(\mathbb{\C}, \mathrm{\textnormal{e}}^{- \vert z \vert^2})$\\
donc 
$$ (\alpha \bar{\partial}^{k}  + c)T_k(\alpha f) = \alpha f$$ avec 
 $$\vert \vert T_k(\alpha f) \vert \vert_\varphi \leq \frac{\vert \alpha \vert^2}{(k!)} \vert \vert f \vert \vert_\varphi.$$
 Ainsi
 $$ (\alpha  \bar{\partial}^{k}  + c)T_k = I$$ avec 
 $$\vert \vert T_k \vert \vert_\varphi \leq \frac{1}{(k!)}.$$
\end{proof}
\begin{thm}\label{t1}
Soit $\varphi$ une fonction lisse et strictement positive sur $\mathbb{C}$ avec $\Delta \varphi > 0$. Pour tout $ f \in L^2(\mathbb{\C}, \mathrm{\textnormal{e}}^{-\varphi})$ telle que $$\frac{f}{\sqrt{\Delta \varphi }}  \in L^2(\mathbb{\C}, \mathrm{\textnormal{e}}^{-\varphi}),$$ il existe une solution faible $ u \in  L^2(\mathbb{\C}, \mathrm{\textnormal{e}}^{-\varphi})$ de l'équation
$$\alpha  \bar{\partial}^1 u + cu = \alpha f$$ avec $\vert \alpha \vert \geq 1$ et l'estimation de la norme
$$\int_\mathbb{C} \vert u \vert^2 \mathrm{\textnormal{e}}^{-\varphi} d\sigma \leq 4  \vert \alpha \vert^2 \int_\mathbb{C} \frac{\vert f \vert^2}{\Delta \varphi}  \mathrm{\textnormal{e}}^{-\varphi} d\sigma.$$
\end{thm}
\begin{proof}
Soit $ f \in L^2(\mathbb{\C}, \mathrm{\textnormal{e}}^{-\varphi})$ telle que  $$\frac{f}{\sqrt{\Delta \varphi }}  \in L^2(\mathbb{\C}, \mathrm{\textnormal{e}}^{-\varphi}),$$
d'après le théorème $1.2$ de \cite{2}, il existe une solution faible $u \in L^2(\mathbb{\C}, \mathrm{\textnormal{e}}^{-\varphi})$ de l'équation
$$  \bar{\partial}^1 u + c' u =  f$$ o\`u $c'= \frac{c}{\alpha}$ avec l'estimation de la norme 
$$\int_\mathbb{C} \vert u \vert^2 \mathrm{\textnormal{e}}^{-\varphi} d\sigma \leq 4  \int_\mathbb{C} \frac{\vert f \vert^2}{\Delta \varphi}  \mathrm{\textnormal{e}}^{-\varphi} d\sigma.$$
donc $$\alpha  \bar{\partial}^1 u + cu = \alpha f.$$
Or $\vert \alpha \vert \geq 1$ donc
$$4  \int_\mathbb{C} \frac{\vert f \vert^2}{\Delta \varphi}  \mathrm{\textnormal{e}}^{-\varphi} d\sigma \leq  4  \vert \alpha \vert^2 \int_\mathbb{C} \frac{\vert f \vert^2}{\Delta \varphi}  \mathrm{\textnormal{e}}^{-\varphi} d\sigma.$$
Ainsi
$$\int_\mathbb{C} \vert u \vert^2 \mathrm{\textnormal{e}}^{-\varphi} d\sigma \leq 4 \vert \alpha \vert^2 \int_\mathbb{C} \frac{\vert f \vert^2}{\Delta \varphi }  \mathrm{\textnormal{e}}^{-\varphi} d\sigma.$$
\end{proof}

\subsection{Quelques conséquences du cas de l'opérateur $\alpha  \bar{\partial}^{k} + c$}

Soient $\lambda > 0$ et $z_0 \in \mathbb{C}$. Posons $\varphi = \lambda \vert z - z_0 \vert^2$, on obtient le corollaire du Théorème \ref{q} suivant:
\begin{coro} \label{lam2}
Soit $f \in L^2(\mathbb{\C}, \mathrm{\textnormal{e}}^{-\lambda \vert z - z_0 \vert^2})$, il existe une solution faible $u \in L^2(\mathbb{\C}, \mathrm{\textnormal{e}}^{-\lambda \vert z - z_0 \vert^2})$ de l'équation
$$\alpha \bar{\partial}^{k} u + cu = \alpha f$$ avec $\vert \alpha \vert \geq 1$ et l'estimation de la norme
$$\int_\mathbb{C} \vert u \vert^2 \mathrm{\textnormal{e}}^{- \lambda \vert z - z_0 \vert^2} d\sigma \leq \frac{\vert \alpha \vert^2}{\lambda^k k!} \int_\mathbb{C} \vert f \vert^2 \mathrm{\textnormal{e}}^{- \lambda \vert z - z_0 \vert^2} d\sigma$$
\end{coro}
\begin{proof}
Soit $f \in L^2(\mathbb{\C}, \mathrm{\textnormal{e}}^{-\lambda \vert z - z_0 \vert^2})$, on a $$\int_\mathbb{C} \vert f \vert^2 \mathrm{\textnormal{e}}^{- \lambda \vert z - z_0 \vert^2} d\sigma < + \infty.$$
Posons $z = \frac{w}{\lambda} + z_0$ et $g(w)= f(z)= f(\frac{w}{\lambda} + z_0)$ alors
$$\frac{1}{\lambda} \int_\mathbb{C} \vert g \vert^2 \mathrm{\textnormal{e}}^{- \vert w \vert^2} d\sigma(w) < + \infty.$$
donc $g \in L^2(\mathbb{\C}, \mathrm{\textnormal{e}}^{-\vert w \vert^2})$. D'après le Théorème \ref{q}, il existe une solution faible $v \in L^2(\mathbb{\C}, \mathrm{\textnormal{e}}^{-\vert w \vert^2})$ de l'équation
$$\alpha \bar{\partial}^{k} v +  \frac{c}{(\sqrt{\lambda})^k}v = \alpha g$$ avec l'estimation de la norme
$$\int_\mathbb{C} \vert v \vert^2 \mathrm{\textnormal{e}}^{- \vert w \vert^2} d\sigma(w) \leq \frac{\vert \alpha \vert^2}{k!} \int_\mathbb{C} \vert g \vert^2 \mathrm{\textnormal{e}}^{- \vert w \vert^2} d\sigma(w).$$
Posons $u(z) = \frac{1}{(\sqrt{\lambda})^k} v(w) = \frac{1}{(\sqrt{\lambda})^k} v(\sqrt{\lambda} (z - z_0))$.\\
Alors 
$$\alpha \bar{\partial}^{k} u + cu = \alpha f$$ avec l'estimation de la norme
$$\int_\mathbb{C} \vert u \vert^2 \mathrm{\textnormal{e}}^{- \lambda \vert z - z_0 \vert^2} d\sigma \leq \frac{\vert \alpha \vert^2}{\lambda^k k!} \int_\mathbb{C} \vert f \vert^2 \mathrm{\textnormal{e}}^{- \lambda \vert z - z_0 \vert^2} d\sigma.$$
\end{proof}

Comme conséquence du Corollaire \ref{lam2} on a
\begin{coro} \label{4}
Soit $U \in \mathbb{C}$ un ouvert borné.\\
Soit $f \in L^2(U)$, il existe une solution faible $u \in L^2(U)$ de l'équation
$$\alpha  \bar{\partial}^{k} u + cu = \alpha f$$ avec $\vert \alpha \vert \geq 1$ et
$$\vert \vert u \vert \vert_{L^2(U)} \leq \Big(\frac{{\textnormal{e}}^{\vert U \vert^2}\vert \alpha \vert^2}{k!}\Big) \vert \vert f \vert \vert_{L^2(U)}$$ 
o\`u $\vert U \vert $ est le diamètre de $U$.
\end{coro}
\begin{proof}
Soient $z_0 \in U$ et $f \in L^2(U)$, on a
 \[ 
 \tilde{f}(z) = \left \{
 \begin{array}{rl}
 f(z) & \mbox{ si } z \in U \\
 0 & \mbox{ si } z \in \mathbb{C} \setminus U
 \end{array}
 \right.
 \]
 Donc $\tilde{f} \in L^2(\mathbb{C}) \subset L^2(\mathbb{\C}, \mathrm{\textnormal{e}}^{-\vert z - z_0 \vert^2})$. D'après le Corollaire \ref{lam2}, il existe une solution faible $\tilde{u} \in  L^2(\mathbb{\C}, \mathrm{\textnormal{e}}^{-\vert z - z_0 \vert^2})$ de l'équation $$\alpha \bar{\partial}^{k} \tilde{u} + c\tilde{u} = \alpha \tilde{f}$$ avec
 $$\int_\mathbb{C} \vert \tilde{u} \vert^2 \mathrm{\textnormal{e}}^{-  \vert z - z_0 \vert^2} d\sigma \leq \frac{\vert \alpha \vert^2}{k!} \int_\mathbb{C} \vert \tilde{f} \vert^2 \mathrm{\textnormal{e}}^{-  \vert z - z_0 \vert^2} d\sigma$$
 donc
 $$\int_\mathbb{C} \vert \tilde{u} \vert^2 \mathrm{\textnormal{e}}^{-  \vert z - z_0 \vert^2} d\sigma \leq \frac{\vert \alpha \vert^2}{k!} \int_U \vert f \vert^2  d\sigma$$
 or $$\int_\mathbb{C} \vert \tilde{u} \vert^2 \mathrm{\textnormal{e}}^{-  \vert z - z_0 \vert^2} d\sigma \geq \mathrm{\textnormal{e}}^{-\vert U \vert^2} \int_U \vert \tilde{u} \vert^2  d\sigma$$
 ainsi
 $$\int_U \vert \tilde{u} \vert^2  d\sigma \leq \Big(\frac{{\textnormal{e}}^{\vert U \vert^2}\vert \alpha \vert^2}{k!}\Big) \int_U \vert f \vert^2  d\sigma$$
 or $\tilde{u}_{\mid U} = u$ alors on a
 $$\alpha  \bar{\partial}^{k} u + cu = \alpha f$$ avec 
$$\vert \vert u \vert \vert_{L^2(U)} \leq \Big(\frac{{\textnormal{e}}^{\vert U \vert^2}\vert \alpha \vert^2}{k!}\Big)\vert \vert f \vert \vert_{L^2(U)}.$$ 
\end{proof}

\subsection{L'opérateur $\alpha \partial^{k} \bar{\partial}^{k}  + c$}

\begin{proof}{(Théorème \ref{p})}
Soit $f \in L^2(\mathbb{\C}, \mathrm{\textnormal{e}}^{- \vert z \vert^2})$, d'après le théorème $1.1$ de \cite{4}, il existe une solution faible $u \in L^2(\mathbb{\C}, \mathrm{\textnormal{e}}^{- \vert z \vert^2})$ de l'équation
$$ \partial^k \bar{\partial}^{k} u + c' u =  f$$ o\`u $c'= \frac{c}{\alpha}$ avec l'estimation de la norme
$$\int_\mathbb{C} \vert u \vert^2 \mathrm{\textnormal{e}}^{- \vert z \vert^2} d\sigma \leq \frac{1}{(k!)^2} \int_\mathbb{C} \vert f \vert^2 \mathrm{\textnormal{e}}^{- \vert z \vert^2} d\sigma$$ 
donc $$\alpha \partial^k \bar{\partial}^{k} u + cu = \alpha f.$$
Or $\vert \alpha \vert \geq 1$ donc
$$\frac{1}{(k!)^2} \int_\mathbb{C} \vert f \vert^2 \mathrm{\textnormal{e}}^{- \vert z \vert^2} d\sigma \leq \frac{\vert \alpha \vert^2}{(k!)^2} \int_\mathbb{C} \vert f \vert^2 \mathrm{\textnormal{e}}^{- \vert z \vert^2} d\sigma$$
Ainsi
$$\int_\mathbb{C} \vert u \vert^2 \mathrm{\textnormal{e}}^{- \vert z \vert^2} d\sigma \leq \frac{\vert \alpha \vert^2}{(k!)^2} \int_\mathbb{C} \vert f \vert^2 \mathrm{\textnormal{e}}^{- \vert z \vert^2} d\sigma$$
\end{proof}
\begin{theorem}
Il existe un opérateur linéaire et borné $$T_k : L^2(\mathbb{\C}, \mathrm{\textnormal{e}}^{- \vert z \vert^2}) \longrightarrow L^2(\mathbb{\C}, \mathrm{\textnormal{e}}^{- \vert z \vert^2})$$
tel que $$ (\alpha \partial^k \bar{\partial}^{k}  + c) T_k = I$$ avec $\vert \alpha \vert \geq 1$ et l'estimation de la norme 
$$\vert \vert T_k \vert \vert_\varphi \leq \frac{1}{(k!)^2}$$ o\`u $\vert \vert T_k \vert \vert_\varphi$ est la norme de $T_k$ dans $L^2(\mathbb{\C}, \mathrm{\textnormal{e}}^{- \vert z \vert^2})$.
\end{theorem}
\begin{proof}
Soit $f \in L^2(\mathbb{\C}, \mathrm{\textnormal{e}}^{- \vert z \vert^2})$. D'après le théorème \ref{p}, il existe une solution faible $u \in L^2(\mathbb{\C}, \mathrm{\textnormal{e}}^{- \vert z \vert^2})$ telle que
$$\alpha \partial^k \bar{\partial}^{k} u + cu = \alpha f$$  avec l'estimation de la norme 
 $$\vert \vert u \vert \vert_\varphi \leq \frac{\vert \alpha \vert^2}{(k!)^2} \vert \vert f \vert \vert_\varphi$$
 Or $\alpha f \in L^2(\mathbb{\C}, \mathrm{\textnormal{e}}^{- \vert z \vert^2})$\\
donc 
$$ (\alpha \partial^k \bar{\partial}^{k}  + c)T_k(\alpha f) = \alpha f$$ avec 
 $$\vert \vert T_k(\alpha f) \vert \vert_\varphi \leq \frac{\vert \alpha \vert^2}{(k!)^2} \vert \vert f \vert \vert_\varphi.$$
 Ainsi
 $$ (\alpha \partial^k \bar{\partial}^{k}  + c)T_k = I$$ avec 
 $$\vert \vert T_k \vert \vert_\varphi \leq \frac{1}{(k!)^2}.$$
\end{proof}
\begin{thm}\label{t1}
Soit $\varphi$ une fonction lisse et strictement positive sur $\mathbb{C}$ avec $\Delta(\mathrm{\textnormal{e}}^{\varphi} \Delta \mathrm{\textnormal{e}}^{-\varphi})> 0$. Pour tout $ f \in L^2(\mathbb{\C}, \mathrm{\textnormal{e}}^{-\varphi})$ telle que $$\frac{f}{\sqrt{\Delta(\mathrm{\textnormal{e}}^{\varphi} \Delta \mathrm{\textnormal{e}}^{-\varphi })}}  \in L^2(\mathbb{\C}, \mathrm{\textnormal{e}}^{-\varphi }),$$ il existe une solution faible $ u \in  L^2(\mathbb{\C}, \mathrm{\textnormal{e}}^{-\varphi})$ de l'équation
$$\alpha \partial^1 \bar{\partial}^1 u + cu = \alpha f$$ avec $\vert \alpha \vert \geq 1$ et l'estimation de la norme
$$\int_\mathbb{C} \vert u \vert^2 \mathrm{\textnormal{e}}^{-\varphi} d\sigma \leq 16  \vert \alpha \vert^2 \int_\mathbb{C} \frac{\vert f \vert^2}{\Delta(\mathrm{\textnormal{e}}^{\varphi} \Delta \mathrm{\textnormal{e}}^{-\varphi})}  \mathrm{\textnormal{e}}^{-\varphi} d\sigma.$$
\end{thm}
\begin{proof}
Soit $ f \in L^2(\mathbb{\C}, \mathrm{\textnormal{e}}^{-\varphi})$ telle que $$\frac{f}{\sqrt{\Delta(\mathrm{\textnormal{e}}^{\varphi} \Delta \mathrm{\textnormal{e}}^{-\varphi})}}  \in L^2(\mathbb{\C}, \mathrm{\textnormal{e}}^{-\varphi}),$$
d'après le théorème $1.2$ de \cite{4}, il existe une solution faible $u \in L^2(\mathbb{\C}, \mathrm{\textnormal{e}}^{-\varphi})$ de l'équation
$$ \partial^1 \bar{\partial}^1 u + c' u =  f$$ o\`u $c'= \frac{c}{\alpha}$ avec l'estimation de la norme 
$$\int_\mathbb{C} \vert u \vert^2 \mathrm{\textnormal{e}}^{-\varphi } d\sigma \leq 16  \int_\mathbb{C} \frac{\vert f \vert^2}{\Delta(\mathrm{\textnormal{e}}^{\varphi} \Delta \mathrm{\textnormal{e}}^{-\varphi})}  \mathrm{\textnormal{e}}^{-\varphi} d\sigma.$$
donc $$\alpha \partial^1 \bar{\partial}^1 u + cu = \alpha f.$$
Or $\vert \alpha \vert \geq 1$ donc
$$16  \int_\mathbb{C} \frac{\vert f \vert^2}{\Delta(\mathrm{\textnormal{e}}^{\varphi} \Delta \mathrm{\textnormal{e}}^{-\varphi})}  \mathrm{\textnormal{e}}^{-\varphi} d\sigma \leq  16  \vert \alpha \vert^2 \int_\mathbb{C} \frac{\vert f \vert^2}{\Delta(\mathrm{\textnormal{e}}^{\varphi} \Delta \mathrm{\textnormal{e}}^{-\varphi})}  \mathrm{\textnormal{e}}^{-\varphi} d\sigma.$$
Ainsi
$$\int_\mathbb{C} \vert u \vert^2 \mathrm{\textnormal{e}}^{-\varphi} d\sigma \leq 16  \vert \alpha \vert^2 \int_\mathbb{C} \frac{\vert f \vert^2}{\Delta(\mathrm{\textnormal{e}}^{\varphi} \Delta \mathrm{\textnormal{e}}^{-\varphi})}  \mathrm{\textnormal{e}}^{-\varphi} d\sigma.$$
\end{proof}
\subsection{Quelques conséquences du cas de l'opérateur $\alpha \partial^{k}  \bar{\partial}^{k} + c$}

Soient $\lambda > 0$ et $z_0 \in \mathbb{C}$. Posons $\varphi = \lambda^2 \vert z - z_0 \vert^2$, on obtient le corollaire du Théorème \ref{p} suivant:
\begin{coro} \label{lam}
Soit $f \in L^2(\mathbb{\C}, \mathrm{\textnormal{e}}^{-\lambda^2 \vert z - z_0 \vert^2})$, il existe une solution faible $u \in L^2(\mathbb{\C}, \mathrm{\textnormal{e}}^{-\lambda^2 \vert z - z_0 \vert^2})$ de l'équation
$$\alpha \partial^k \bar{\partial}^{k} u + cu = \alpha f$$ avec $\vert \alpha \vert \geq 1$ et l'estimation de la norme
$$\int_\mathbb{C} \vert u \vert^2 \mathrm{\textnormal{e}}^{- \lambda^2 \vert z - z_0 \vert^2} d\sigma \leq \frac{\vert \alpha \vert^2}{(\lambda^k k!)^2} \int_\mathbb{C} \vert f \vert^2 \mathrm{\textnormal{e}}^{- \lambda^2 \vert z - z_0 \vert^2} d\sigma$$
\end{coro}
\begin{proof}
Soit $f \in L^2(\mathbb{\C}, \mathrm{\textnormal{e}}^{-\lambda^2 \vert z - z_0 \vert^2})$, on a $$\int_\mathbb{C} \vert f \vert^2 \mathrm{\textnormal{e}}^{- \lambda^2 \vert z - z_0 \vert^2} d\sigma < + \infty.$$
Posons $z = \frac{w}{\lambda} + z_0$ et $g(w)= f(z)= f(\frac{w}{\lambda} + z_0)$ alors
$$\frac{1}{\lambda^2} \int_\mathbb{C} \vert g \vert^2 \mathrm{\textnormal{e}}^{- \vert w \vert^2} d\sigma(w) < + \infty.$$
donc $g \in L^2(\mathbb{\C}, \mathrm{\textnormal{e}}^{-\vert w \vert^2})$. D'après le Théorème \ref{p}, il existe une solution faible $v \in L^2(\mathbb{\C}, \mathrm{\textnormal{e}}^{-\vert w \vert^2})$ de l'équation
$$\alpha \partial^k \bar{\partial}^{k} v + \frac{c}{\lambda^k}v = \alpha g$$ avec l'estimation de la norme
$$\int_\mathbb{C} \vert v \vert^2 \mathrm{\textnormal{e}}^{- \vert w \vert^2} d\sigma(w) \leq \frac{\vert \alpha \vert^2}{( k!)^2} \int_\mathbb{C} \vert g \vert^2 \mathrm{\textnormal{e}}^{- \vert w \vert^2} d\sigma(w).$$
Posons $u(z) = \frac{1}{\lambda^k} v(w) = \frac{1}{\lambda^k} v(\lambda (z - z_0))$.\\
Alors 
$$\alpha \partial^k \bar{\partial}^{k} u + cu = \alpha f$$ avec l'estimation de la norme
$$\int_\mathbb{C} \vert u \vert^2 \mathrm{\textnormal{e}}^{- \lambda^2 \vert z - z_0 \vert^2} d\sigma \leq \frac{\vert \alpha \vert^2}{(\lambda^k k!)^2} \int_\mathbb{C} \vert f \vert^2 \mathrm{\textnormal{e}}^{- \lambda^2 \vert z - z_0 \vert^2} d\sigma.$$
\end{proof}

Comme conséquence du Corollaire \ref{lam} on a
\begin{coro} \label{3}
Soit $U \in \mathbb{C}$ un ouvert borné.\\
Soit $f \in L^2(U)$, il existe une solution faible $u \in L^2(U)$ de l'équation
$$\alpha \partial^k \bar{\partial}^{k} u + cu = \alpha f$$ avec $\vert \alpha \vert \geq 1$ et
$$\vert \vert u \vert \vert_{L^2(U)} \leq \Big(\frac{{\textnormal{e}}^{\vert U \vert^2}\vert \alpha \vert^2}{(k!)^2}\Big) \vert \vert f \vert \vert_{L^2(U)}$$ 
o\`u $\vert U \vert $ est le diamètre de $U$.
\end{coro}
\begin{proof}
Soient $z_0 \in U$ et $f \in L^2(U)$, on a
 \[ 
 \tilde{f}(z) = \left \{
 \begin{array}{rl}
 f(z) & \mbox{ si } z \in U \\
 0 & \mbox{ si } z \in \mathbb{C} \setminus U
 \end{array}
 \right.
 \]
 Donc $\tilde{f} \in L^2(\mathbb{C}) \subset L^2(\mathbb{\C}, \mathrm{\textnormal{e}}^{-\vert z - z_0 \vert^2})$. D'après le Corollaire \ref{lam}, il existe une solution faible $\tilde{u} \in L^2(\mathbb{\C}, \mathrm{\textnormal{e}}^{-\vert z - z_0 \vert^2})$ de l'équation $$\alpha \partial^k \bar{\partial}^{k} \tilde{u} + c\tilde{u} = \alpha \tilde{f}$$ avec
 $$\int_\mathbb{C} \vert \tilde{u} \vert^2 \mathrm{\textnormal{e}}^{-  \vert z - z_0 \vert^2} d\sigma \leq \frac{\vert \alpha \vert^2}{(k!)^2} \int_\mathbb{C} \vert \tilde{f} \vert^2 \mathrm{\textnormal{e}}^{-  \vert z - z_0 \vert^2} d\sigma$$
 donc
 $$\int_\mathbb{C} \vert \tilde{u} \vert^2 \mathrm{\textnormal{e}}^{-  \vert z - z_0 \vert^2} d\sigma \leq \frac{\vert \alpha \vert^2}{(k!)^2} \int_U \vert f \vert^2  d\sigma$$
 or $$\int_\mathbb{C} \vert \tilde{u} \vert^2 \mathrm{\textnormal{e}}^{-  \vert z - z_0 \vert^2} d\sigma \geq \mathrm{\textnormal{e}}^{-\vert U \vert^2} \int_U \vert \tilde{u} \vert^2  d\sigma$$
 ainsi
 $$\int_U \vert \tilde{u} \vert^2  d\sigma \leq \Big(\frac{{\textnormal{e}}^{\vert U \vert^2}\vert \alpha \vert^2}{(k!)^2}\Big) \int_U \vert f \vert^2  d\sigma$$
 or $\tilde{u}_{\mid U} = u$ alors on a
 $$\alpha \partial^k \bar{\partial}^{k} u + cu = \alpha f$$ avec 
$$\vert \vert u \vert \vert_{L^2(U)} \leq \Big(\frac{{\textnormal{e}}^{\vert U \vert^2}\vert \alpha \vert^2}{(k!)^2}\Big)\vert \vert f \vert \vert_{L^2(U)}.$$ 
\end{proof}

\end{document}